\documentclass{article}
\usepackage{amsfonts}
\usepackage{amsthm}
\usepackage{amsmath}
\usepackage{tikz}

\DeclareMathOperator{\Sym}{Sym}

\DeclareMathOperator{\Hom}{Hom}
\DeclareMathOperator{\tr}{tr}
\newcommand{\id}{\mathrm{id}}

\newtheorem{theorem}{Theorem}[section]
\newtheorem{lemma}[theorem]{Lemma}

\newtheorem{definition}[theorem]{Definition}

\title{The Brauer category $\mathcal{B}(2)$ has \\ principal graph $D_\infty$}
\author{Stephen Bigelow}
\date{June 2026}

\begin{document}

\maketitle

\begin{abstract}
We show that the subfactor planar algebra with principal graph $D_\infty$
is the Brauer planar algebra with bubble constant $\delta=2$.  The Brauer
algebra is similar to the Temperley-Lieb algebra, but with virtual
crossings.  At $\delta=2$, it relates to the category of representations
of the orthogonal group $O(2)$.  We work over any commutative ring with
$1/2$.  Its principal graph encodes information about the corresponding
monoidal category.
\end{abstract}

\section{Introduction}
\label{sec:intro}

This paper grew out of my interest in the skein theory of subfactor planar
algebras.  Skein theory is the approach to studying the planar algebra in
terms of linear combinations of diagrams, modulo certain local relations.

Jones introduced planar algebras in the 1990s, now published in
\cite{MR4374438}, as a way to encode the standard invariant of a
type $\mathrm{II}_1$ subfactor.  The \emph{principal graph} and the
\emph{index} are two other invariants of subfactors.  Hyperfinite
$\mathrm{II}_1$ subfactors of index less than or equal to $4$ have an
``ADE'' classification \cite[Section~5.2]{MR1278111}.

My own introduction to this subject was \cite{MR2559686}, which finds
a skein theory for the subfactor planar algebra with principal graph
$D_{2n}$.  The aim of this paper is to do something similar for the
affine Dynkin diagram $D_\infty$.
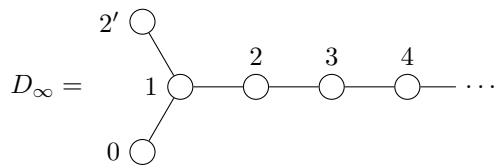
$$
D_\infty=
\begin{tikzpicture}[baseline=-0.5ex]
\node (BASE) at (240:1) [shape=circle,draw]{};
\node (Q) at (120:1) [shape=circle,draw]{};
\node (X) at (0,0) [shape=circle,draw]{};
\node (P) at (1,0) [shape=circle,draw]{};
\node (P3) at (2,0) [shape=circle,draw]{};
\node (P4) at (3,0) [shape=circle,draw]{};
\node [left] at (BASE.west) {$0$};
\node [left] at (X.west) {$1$};
\node [left] at (Q.west) {$2'$};
\node [above] at (P.north) {$2$};
\node [above] at (P3.north) {$3$};
\node [above] at (P4.north) {$4$};
\node (ETC) at (4,0) {\dots};
\draw (BASE)--(X);
\draw (Q)--(X);
\draw (X)--(P)--(P3)--(P4)--(ETC);
\end{tikzpicture}
$$
Here, we have numbered the vertices for future reference.  Vertex
$0$ is the base vertex.

It turns out that the planar algebra in this case is the reduced Brauer
algebra at the value $\delta=2$.  Here, $\delta$ is the value of a bubble,
and the index is $\delta^2$.

The $D_\infty$ subfactor was studied by Bisch and Jones \cite{MR1437496}.
It is a special case of a Fuss-Catalan subfactor, and they give generators
and relations that are equivalent to ours.  What we will call the leaf
relation in Lemma~\ref{lem:leaf} is a version of their characterisation
of a Jones projection onto an intermediate subfactor as a biprojection
\cite{MR1733737}.

From the other direction, the Brauer algebra is well known to
representation theorists.  For example, Lehrer and Zhang \cite{MR3420509}
describe the category of representations of the orthogonal group $O(n)$
as a quotient of the Brauer category.  We are concerned with $O(2)$,
the first case that is not completely trivial.

Some experts already know a diagrammatic description of the $D_\infty$
planar algebra in terms of generators and relations.  One was given in
Molander's PhD thesis.

So, in some sense, the example studied in this paper is already well
understood.  However the connection between the principal graph $D_\infty$
and the Brauer algebra does not seem to have been noted.  I also tried
to make the results as generally applicable as possible, working over
any ring that has $1/2$, and not using the shading or unitarity axioms
from Jones's original definition of a subfactor planar algebra.

\section{The Brauer category}
\label{sec:brauer}

We define the Brauer category $\mathcal{B}$, and the reduced Brauer
category $\overline{\mathcal{B}}$.  They depend on a commutative unital
ring $R$, and a scalar $\delta$.  We take $R$ to be any commutative
unital ring with $1/2$, and we will always assume $\delta=2$.  We use the
terms Brauer planar algebra and Brauer category almost interchangeably.
Thought of as a category, diagrams drawn in a rectangle are morphisms,
and planar ways of connecting diagrams correspond to operations in
the category.

A \emph{Brauer diagram} consists of finitely many strands drawn in a
rectangle, having their endpoints on the top and bottom edge of the
rectangle.  Two such diagrams are considered equal if they connect the
same pairs of endpoints on the boundary of the rectangle.

The objects of $\mathcal{B}$ are the natural numbers, including zero.
The morphisms from $m$ to $n$ are formal $R$-linear combinations of Brauer
diagrams that have $m$ endpoints on the bottom and $n$ on the top.  If $m$
and $n$ do not have the same parity then there are no such diagrams,
and the only morphism from $m$ to $n$ is the zero morphism.

The composition $AB$ of diagrams is given by stacking $A$ on top of $B$,
deleting any closed loops in the resulting diagram, and multiplying by $2$
to the power of the number of loops that were deleted.  Extend this to a
bilinear operation on morphisms.  The identity morphism $\id_n$ is the
diagram consisting of $n$ parallel vertical strands.

There is a tensor product operation that turns $\mathcal{B}$ into a strict
monoidal category.  The tensor product of objects is given by addition.
The tensor product of diagrams $A\otimes B$ be given by placing $B$
to the right of $A$.  Extend this to a bilinear operation on morphisms.

There are also duals that turn $\mathcal{B}$ into a strict pivotal
category.  The objects are all self-dual.  The dual $A^\dagger$ of a
diagram is obtained by rotating $A$ by $180$ degrees.  Extend this to
a linear operation on morphisms.  The evaluation map from
$n \otimes n$ to $0$ is the diagram consisting of $n$ parallel strands
forming a ``rainbow''.  The coevaluation map is dual to the evaluation
map.

An \emph{$n$-box} is just an endomorphism of $n$.  The space of $n$-boxes
is an associative algebra, called the Brauer algebra $\mathcal{B}_n$.
The space of $0$-boxes is isomorphic to $R$, generated by the empty
diagram.

To calculate the \emph{trace} $\tr(A)$ of an $n$-box diagram $A$, first
connect the $n$ endpoints at the top to the $n$ endpoints at the bottom
by disjoint arcs that pass around $A$ to the right.  Then $\tr(A)$ is $2$
to the power of the number of loops in the resulting diagram.  Extend
this to a linear operation on endomorphisms.  Note that $\mathcal{B}$
is \emph{spherical}, meaning that the left and right traces are the same.

We now add one more relation.  Let $A_3$ denote the $3$-box corresponding
to the \emph{antisymmetrizer} of the symmetric group $\Sym_3$.
$$
A_3 =
\begin{tikzpicture}[baseline=-0.5ex,scale=0.15]
\draw (-2,-2) -- (-2,2);
\draw (0,-2) -- (0,2);
\draw (2,-2) -- (2,2);
\end{tikzpicture}
\;-\;
\begin{tikzpicture}[baseline=-0.5ex,scale=0.15]
\draw (-2,-2) -- (-2,2);
\draw (0,-2) .. controls (0,0) and (2,0) .. (2,2);
\draw (2,-2) .. controls (2,0) and (0,0) .. (0,2);
\end{tikzpicture}
\;+\;
\begin{tikzpicture}[baseline=-0.5ex,scale=0.15]
\draw (-2,-2) .. controls (-2,0) and (0,0) .. (0,2);
\draw (0,-2) .. controls (0,0) and (2,0) .. (2,2);
\draw (2,-2) .. controls (2,0) and (-2,0) .. (-2,2);
\end{tikzpicture}
\;-\;
\begin{tikzpicture}[baseline=-0.5ex,scale=0.15]
\draw (-2,-2) .. controls (-2,-1) and (2,-1) .. (2,2);
\draw (0,-2) .. controls (0,-1) and (-1,-1) ..
      (-1,0) .. controls (-1,1) and (0,1) ..  (0,2);
\draw (2,-2) .. controls (2,1) and (-2,1) .. (-2,2);
\end{tikzpicture}
\;+\;
\begin{tikzpicture}[baseline=-0.5ex,scale=0.15]
\draw (-2,-2) .. controls (-2,0) and (2,0) .. (2,2);
\draw (0,-2) .. controls (0,0) and (-2,0) .. (-2,2);
\draw (2,-2) .. controls (2,0) and (0,0) .. (0,2);
\end{tikzpicture}
\;-\;
\begin{tikzpicture}[baseline=-0.5ex,scale=0.15]
\draw (-2,-2) .. controls (-2,0) and (0,0) .. (0,2);
\draw (0,-2) .. controls (0,0) and (-2,0) .. (-2,2);
\draw (2,-2) -- (2,2);
\end{tikzpicture}
$$
Let the \emph{reduced Brauer category} $\overline{\mathcal{B}}$ be
the quotient of $\mathcal{B}$ by $A_3$.  Thus $\overline{\mathcal{B}}$
has the same objects, but we quotient the $\Hom$ spaces by $A_3=0$
and any consequence of this by compositions and tensor products.

The significance of $A_3$ is that it is \emph{negligible}
in $\mathcal{B}$, meaning that the trace of $fA_3$ is zero for any
$3$-box $f$.  Thus the space of $0$-boxes in $\overline{\mathcal{B}}$
is still isomorphic to $R$.

Note that we could also have defined $\overline{\mathcal{B}}$ in terms
of diagrammatic generators and relations.  It is generated by a virtual
crossing, which is a $2$-box.  The relations state that a bubble is $2$,
the virtual crossing has $90$ degree rotational symmetry and satisfies
the three Reidemeister moves, and $A_3=0$.  We use the terminology
\emph{Reidemeister move} even though all crossings are virtual, and
perhaps \emph{detour move} would be more accurate.

\section{Principal graphs}

Let $\mathcal{C}$ be an $R$-linear spherical category whose objects are
the natural numbers, including zero.  Assume the objects all self-dual,
with the tensor product given by addition.

We now explain what we mean when we say $\mathcal{C}$ has a
given principal graph.  It would be more accurate to think of the
principal graph as encoding information about the Cauchy completion of
$\mathcal{C}$, but we will make some definitions so that we can avoid
changing the category.

An \emph{idempotent} is an endomorphism $P$ such that $P^2=P$.  If $P$
and $Q$ are idempotents then the \emph{$\Hom$-space} $\Hom(P,Q)$ is
the $R$-module of morphisms of the form $QfP$.

We say two idempotents $P$ and $Q$ are \emph{isomorphic}, denoted
$P \simeq Q$, if there exist morphisms $f$ and $g$ such that $P=gf$ and
$Q=fg$.  We say an idempotent $P$ is isomorphic to a \emph{direct sum}
of idempotents $Q_1,\dots Q_n$ if $P=P_1+\dots+P_n$ for some idempotents
$P_i$ such that $P_iP_j=0$ whenever $i\neq j$, and $P_i$ is isomorphic
to $Q_i$ for all $i$.

A \emph{brick} is an idempotent $P$ such that $\Hom(P,P)$ is isomorphic to
$R$, generated by $P$.  Two idempotents $P$ and $Q$ are \emph{orthogonal}
if $\Hom(P,Q)=0$.

Let $\Gamma$ be a locally finite connected graph with a fixed base vertex.

\begin{definition}
\label{def:principal}
Let $\mathcal{C}$ and $\Gamma$ be as above.  We say $\mathcal{C}$
has \emph{principal graph} $\Gamma$ if we can assign a brick $P_v$
to each vertex $v$ such that the $P_v$ are pairwise orthogonal, if $v$
is the base vertex then $P_v=\id_0$, and every $P_v \otimes \id_1$ is
isomorphic to the direct sum of $P_w$ over all vertices $w$ adjacent to
$v$ in $\Gamma$.
\end{definition}

\section{The principal graph of the Brauer category}

The aim of this section is to prove the following.

\begin{theorem}
\label{thm:principal}
The reduced Brauer category $\overline{\mathcal{B}}$ at
$\delta=2$ has principal graph $D_\infty$, in the sense of
Definition~\ref{def:principal}.
\end{theorem}

Number the vertices $D_\infty$ as shown in Section~\ref{sec:intro}.
Let $P_0$ be the empty diagram, and let the other idempotents be as
follows.
\[
P_1 = 
\begin{tikzpicture}[baseline=-0.5ex,scale=0.2]
\draw (0,-2) -- (0,2);
\end{tikzpicture}
\;,\quad
P_{2'} = \frac12 \left(
\;
\begin{tikzpicture}[baseline=-0.5ex,scale=0.2]
\draw (-1,-2) -- (-1,2);
\draw (1,-2) -- (1,2);
\end{tikzpicture}
-
\begin{tikzpicture}[baseline=-0.5ex,scale=0.2]
\draw (-1,-2) .. controls (-1,0) and (1,0) .. (1,2);
\draw (1,-2) .. controls (1,0) and (-1,0) .. (-1,2);
\end{tikzpicture}
\;
\right)
,\quad
P_2 = \frac12 \left(
\;
\begin{tikzpicture}[baseline=-0.5ex,scale=0.2]
\draw (-1,-2) -- (-1,2);
\draw (1,-2) -- (1,2);
\end{tikzpicture}
-
\begin{tikzpicture}[baseline=-0.5ex,scale=0.2]
\draw (-1,-2) .. controls (-1,0) and (1,0) .. (1,-2);
\draw (1,2) .. controls (1,0) and (-1,0) .. (-1,2);
\end{tikzpicture}
+
\begin{tikzpicture}[baseline=-0.5ex,scale=0.2]
\draw (-1,-2) .. controls (-1,0) and (1,0) .. (1,2);
\draw (1,-2) .. controls (1,0) and (-1,0) .. (-1,2);
\end{tikzpicture}
\;
\right)
,
\]
\[
P_3 = 
\begin{tikzpicture}[baseline=-0.5ex,scale=0.05]
\node at (-5,9) {$P_2$};
\draw (-12,4) rectangle (2,14);
\node at (5,-9) {$P_2$};
\draw (-2,-14) rectangle (12,-4);
\draw (-10,-20) -- (-10,4);
\draw (-10,14) -- (-10,20);
\draw (0,-20) -- (0,-14);
\draw (0,-4) -- (0,4);
\draw (0,14) -- (0,20);
\draw (10,-20) -- (10,-14);
\draw (10,-4) -- (10,20);
\end{tikzpicture}
,
\quad
P_4 =
\begin{tikzpicture}[baseline=-0.5ex,scale=0.05]
\node at (-10,15) {$P_2$};
\draw (-17,10) rectangle (-3,20);
\node at (0,0) {$P_2$};
\draw (-7,-5) rectangle (7,5);
\node at (10,-15) {$P_2$};
\draw (3,-20) rectangle (17,-10);
\draw (-15,-25) -- (-15,10);
\draw (-15,20) -- (-15,25);
\draw (-5,-25) -- (-5,-5);
\draw (-5,5) -- (-5,10);
\draw (-5,20) -- (-5,25);
\draw (5,-25) -- (5,-20);
\draw (5,-10) -- (5,-5);
\draw (5,5) -- (5,25);
\draw (15,-25) -- (15,-20);
\draw (15,-10) -- (15,25);
\end{tikzpicture}
,
\qquad
\dots
.
\]
These are all self-dual, since they have an obvious rotational symmetry.

We make frequent use of the following relation.  We call it the leaf
relation, for reasons we will explain in the next section.

\begin{lemma}[The leaf relation]
\label{lem:leaf}
$P_{2'}$ satisfies the following relation.
\[
\begin{tikzpicture}[baseline=-0.5ex,scale=0.05]
\node at (-10,0) {$P_{2'}$};
\draw (-17,-5) rectangle (-3,5);
\draw (-15,-20) -- (-15,-5);
\draw (-15,5) -- (-15,20);
\draw (-5,-20) -- (-5,-5);
\draw (-5,5) -- (-5,20);
\draw (5,-20) -- (5,20);
\end{tikzpicture}
=2\;
\begin{tikzpicture}[baseline=-0.5ex,scale=0.05]
\node at (-5,10) {$P_{2'}$};
\draw (-12,5) rectangle (2,15);
\node at (-5,-10) {$P_{2'}$};
\draw (-12,-15) rectangle (2,-5);
\draw (-10,-20) -- (-10,-15);
\draw (-10,-5) -- (-10,5);
\draw (-10,15) -- (-10,20);
\draw (0,-20) -- (0,-15);
\draw (0,-5) .. controls (0,0) and (10,0) .. (10,-20);
\draw (0,5) .. controls (0,0) and (10,0) .. (10,20);
\draw (0,15) -- (0,20);
\end{tikzpicture}
\]
\end{lemma}

\begin{proof}
The antisymmetrizer $A_3$ from Section \ref{sec:brauer} is equal to
the following.
$$
A_3 =
2\;
\begin{tikzpicture}[baseline=-0.5ex,scale=0.05]
\node at (-10,0) {$P_{2'}$};
\draw (-17,-5) rectangle (-3,5);
\draw (-15,-20) -- (-15,-5);
\draw (-15,5) -- (-15,20);
\draw (-5,-20) -- (-5,-5);
\draw (-5,5) -- (-5,20);
\draw (5,-20) -- (5,20);
\end{tikzpicture}
-4\;
\begin{tikzpicture}[baseline=-0.5ex,scale=0.05]
\node at (-5,10) {$P_{2'}$};
\draw (-12,5) rectangle (2,15);
\node at (-5,-10) {$P_{2'}$};
\draw (-12,-15) rectangle (2,-5);
\draw (-10,-20) -- (-10,-15);
\draw (-10,-5) -- (-10,5);
\draw (-10,15) -- (-10,20);
\draw (0,-20) -- (0,-15);
\draw (0,-5) .. controls (0,0) and (10,0) .. (10,20);
\draw (10,-20) .. controls (10,0) and (0,0) .. (0,5);
\draw (0,15) -- (0,20);
\end{tikzpicture}
$$
Set $A_3=0$ and divide by $2$.  Now add a crossing between the top right
and bottom right endpoints, and apply Reidemeister one on the left,
and Reidemeister two on the right.
\end{proof}

I claim that each $P_v$ is \emph{uncappable} in the following sense.
\[
\begin{tikzpicture}[baseline=-0.5ex,scale=0.05]
\node at (0,0) {$P_2$};
\draw (-7,-5) rectangle (7,5);
\draw (-5,-10) -- (-5,-5);
\draw (5,-10) -- (5,-5);
\draw (-5,5) .. controls (-5,15) and (5,15) .. (5,5);
\end{tikzpicture}
=
\begin{tikzpicture}[baseline=-0.5ex,scale=0.05]
\node at (0,0) {$P_{2'}$};
\draw (-7,-5) rectangle (7,5);
\draw (-5,-10) -- (-5,-5);
\draw (5,-10) -- (5,-5);
\draw (-5,5) .. controls (-5,15) and (5,15) .. (5,5);
\end{tikzpicture}
=0,\qquad
\begin{tikzpicture}[baseline=-0.5ex,scale=0.05]
\node at (0,0) {$P_3$};
\draw (-12,-5) rectangle (12,5);
\draw (-10,-10) -- (-10,-5);
\draw (0,-10) -- (0,-5);
\draw (10,-10) -- (10,-5);
\draw (-10,5) .. controls (-10,15) and (0,15) .. (0,5);
\draw (10,5) -- (10,15);
\end{tikzpicture}
=
\begin{tikzpicture}[baseline=-0.5ex,scale=0.05]
\node at (0,0) {$P_3$};
\draw (-12,-5) rectangle (12,5);
\draw (-10,-10) -- (-10,-5);
\draw (0,-10) -- (0,-5);
\draw (10,-10) -- (10,-5);
\draw (-10,5) -- (-10,15);
\draw (0,5) .. controls (0,15) and (10,15) .. (10,5);
\end{tikzpicture}
=0,\qquad\dots.
\]
This is a direct calculation for $P_2$, $P_{2'}$ and $P_3$.
If $v=4,5,\dots$, then any cap on the top of $P_v$ can slide down to
annihilate a copy of $P_3$ inside the definition of $P_v$.

Next we show that every $P_v$ can \emph{eat crossings}
in the following sense
\[
\begin{tikzpicture}[baseline=-0.5ex,scale=0.05]
\node at (0,0) {$P_2$};
\draw (-7,-5) rectangle (7,5);
\draw (-5,-10) -- (-5,-5);
\draw (-5,5) .. controls (-5,10) and (5,10) .. (5,15);
\draw (5,-10) -- (5,-5);
\draw (5,5) .. controls (5,10) and (-5,10) .. (-5,15);
\end{tikzpicture}
=
\begin{tikzpicture}[baseline=-0.5ex,scale=0.05]
\node at (0,0) {$P_2$};
\draw (-7,-5) rectangle (7,5);
\draw (-5,-10) -- (-5,-5);
\draw (-5,5) -- (-5,15);
\draw (5,-10) -- (5,-5);
\draw (5,5) -- (5,15);
\end{tikzpicture}
\;,\qquad
\begin{tikzpicture}[baseline=-0.5ex,scale=0.05]
\node at (0,0) {$P_3$};
\draw (-12,-5) rectangle (12,5);
\draw (-10,-10) -- (-10,-5);
\draw (0,-10) -- (0,-5);
\draw (10,-10) -- (10,-5);
\draw (-10,5) .. controls (-10,10) and (0,10) .. (0,15);
\draw (0,5) .. controls (0,10) and (-10,10) .. (-10,15);
\draw (10,5) -- (10,15);
\end{tikzpicture}
=
\begin{tikzpicture}[baseline=-0.5ex,scale=0.05]
\node at (0,0) {$P_3$};
\draw (-12,-5) rectangle (12,5);
\draw (-10,-10) -- (-10,-5);
\draw (0,-10) -- (0,-5);
\draw (10,-10) -- (10,-5);
\draw (-10,5) -- (-10,15);
\draw (0,5) .. controls (0,10) and (10,10) .. (10,15);
\draw (10,5) .. controls (10,10) and (0,10) .. (0,15);
\end{tikzpicture}
=
\begin{tikzpicture}[baseline=-0.5ex,scale=0.05]
\node at (0,0) {$P_3$};
\draw (-12,-5) rectangle (12,5);
\draw (-10,-10) -- (-10,-5);
\draw (0,-10) -- (0,-5);
\draw (10,-10) -- (10,-5);
\draw (-10,5) -- (-10,15);
\draw (0,5) -- (0,15);
\draw (10,5) -- (10,15);
\end{tikzpicture}
\;,\qquad\dots,
\]
or eat a crossing up to sign in the case of $P_{2'}$
\[
\begin{tikzpicture}[baseline=-0.5ex,scale=0.05]
\node at (0,0) {$P_{2'}$};
\draw (-7,-5) rectangle (7,5);
\draw (-5,-10) -- (-5,-5);
\draw (-5,5) .. controls (-5,10) and (5,10) .. (5,15);
\draw (5,-10) -- (5,-5);
\draw (5,5) .. controls (5,10) and (-5,10) .. (-5,15);
\end{tikzpicture}
=-\;
\begin{tikzpicture}[baseline=-0.5ex,scale=0.05]
\node at (0,0) {$P_{2'}$};
\draw (-7,-5) rectangle (7,5);
\draw (-5,-10) -- (-5,-5);
\draw (-5,5) -- (-5,15);
\draw (5,-10) -- (5,-5);
\draw (5,5) -- (5,15);
\end{tikzpicture}.
\]
This is a direct calculation for $P_2$ and $P_{2'}$.  For $v=3,4,\dots$,
we have that $P_v$ can eat a crossing at its top left, since its
definition has a copy of $P_2$ there.  It can eat a crossing at any
position on its top by the following lemma.

\begin{lemma}
If $Q$ is a morphism that is uncappable and can eat a crossing at its
top left, then it can also eat crossings at any position on its top.
\end{lemma}

\begin{proof}
First assume $Q$ is a $3$-box.  Attach a copy of $P_{2'}$ to the top
right of $Q$.  Now apply a $180$ degree rotation of Lemma~\ref{lem:leaf}
to replace that $P_{2'}$ by a diagram that has two copies of $P_{2'}$.
Expand out the lower copy of $P_{2'}$ into two terms.  One of these
puts a cap on the left of $Q$, and the other puts a cap on the right of
$Q$ after $Q$ eats a crossing on its top left.  We conclude that attaching
$P_{2'}$ to the top left of $Q$ gives zero, and hence that $Q$ can eat
a crossing on its top right.

In general, we can use a similar argument to show that if $Q$ can eat a
crossing at a given position on its top, then it can also eat a crossing
at the next position to the right.  Work from left to right to show $Q$
can eat a crossing at any position on its top.
\end{proof}

Note that the above lemma would also hold if we are given that $Q$
can eat a crossing on its top right, or at any given position on its top.

In the Temperley-Lieb algebra, the Jones-Wenzl idempotents are
the unique uncappable elements, up to multiplication by a scalar.
In $\overline{\mathcal{B}}$, the idempotents $P_2,P_3,\dots$ have a
similar uniqueness result.

\begin{lemma}
\label{lem:uniqueuptoscalar}
If $Q$ is an uncappable $n$-box that can eat a crossing at its top left
then $Q$ is a scalar multiple of $P_n$.
\end{lemma}

\begin{proof}
Consider $P_nQ$.  On the one hand, $Q$ is uncappable and eats crossings,
so it can eat every copy of $P_2$ in the definition of $P_n$.
Thus $P_nQ=Q$.

On the other hand, by its rotational symmetry, $P_n$ is uncuppable
and eats crossings from below.  Thus we could expand $Q$ into a linear
combination of Brauer diagrams, let $P_n$ eat crossings and annihilate
any terms with caps, and reduce $P_nQ$ to a scalar multiple of $P_n$.
\end{proof}

We consider $P_0$ and $P_1$ to be vacuously uncappable, since there
is no room for a cap.  By rotational symmetry, each $P_v$ is also
\emph{uncappable}, and can eat crossings from below, up to sign in the
case of $P_{2'}$.

\begin{proof}[Proof of Theorem \ref{thm:principal}]
We must show that each $P_v$ is in fact an idempotent.  In $P_v^2$, expand
out the top copy of $P_v$, and use the fact that the bottom copy of $P_v$
is uncappable and eats crossings.  The result is $P_v$ in every case.

We also have to check that $P_v$ is non-zero and not a torsion element.
An easy way to do this is to calculate the trace of $P_v$.  This is $1$
for $P_0$ and $P_{2'}$, and $2$ for $P_1,P_2,P_3,\dots$.

Now we show that every morphism of the form $P_vfP_v$ is a scalar
multiple of $P_v$.  We can assume $f$ is a single Brauer diagram.
Let each $P_v$ eat crossings.  Either we reduce to a diagram with a cup
or cap, or the identity diagram.  The result is zero, or $P_v^2=P_v$,
or $-P_v$ if $v=2'$ and $f$ was a single crossing.

A similar argument can be used to show that every morphism of the form
$P_wfP_v$ is zero when $v \neq w$.  Eat crossings and reduce to a diagram
with a cup or cap, or the identity diagram.  A cup or cap gives zero.
The identity diagram can only happen when $P_v$ and $P_w$ are in the
same $n$-box space, in which case we have
$$P_2 P_{2'} = P_{2'} P_2 = 0.$$

It remains to show that show that, for all $v$, $P_v \otimes \id_1$
is isomorphic to the direct sum of $P_w$ over all $w$ adjacent to $v$.

For $v=0$, we have $P_0 \otimes \id_1 = P_1$.
For $v=1$, we have $P_1 \otimes \mathrm{id}_1 = P_2 + P_{2'} + E$,
where $E$ is the Jones idempotent
$$
E=
\frac12\;
\begin{tikzpicture}[baseline=-0.5ex,scale=0.05]
\draw (-5,-5) .. controls (-5,0) and (5,0) .. (5,-5);
\draw (-5,5) .. controls (-5,0) and (5,0) .. (5,5);
\end{tikzpicture}
\simeq
\frac12\;
\begin{tikzpicture}[baseline=-0.5ex,scale=0.05]
\draw (-5,0) .. controls (-5,5) and (5,5) .. (5,0)
             .. controls (5,-5) and (-5,-5) .. (-5,0);
\end{tikzpicture}
= P_0.
$$
For $v=2'$, the leaf relation implies
\[
P_{2'} \otimes \mathrm{id}_1
=2\;
\begin{tikzpicture}[baseline=-0.5ex,scale=0.05]
\node at (-5,10) {$P_{2'}$};
\draw (-12,5) rectangle (2,15);
\node at (-5,-10) {$P_{2'}$};
\draw (-12,-15) rectangle (2,-5);
\draw (-10,-20) -- (-10,-15);
\draw (-10,-5) -- (-10,5);
\draw (-10,15) -- (-10,20);
\draw (0,-20) -- (0,-15);
\draw (0,-5) .. controls (0,0) and (10,0) .. (10,-20);
\draw (0,5) .. controls (0,0) and (10,0) .. (10,20);
\draw (0,15) -- (0,20);
\end{tikzpicture}
\simeq
2\;
\begin{tikzpicture}[baseline=-0.5ex,scale=0.05]
\node at (-5,10) {$P_{2'}$};
\draw (-12,5) rectangle (2,15);
\node at (-5,-10) {$P_{2'}$};
\draw (-12,-15) rectangle (2,-5);
\draw (-10,-20) -- (-10,-15);
\draw (-10,-5) -- (-10,5);
\draw (-10,15) -- (-10,20);
\draw (0,-15) .. controls (0,-20) and (10,-20) .. (10,0)
              .. controls (10,20) and (0,20) .. (0,15);
\draw (0,-5) -- (0,5);
\end{tikzpicture}
=P_1.
\]
Finally, for $v=2,3,\dots$, I claim we have
$$P_v \otimes \mathrm{id}_1 = P_{v+1} + F,$$
where $F$ is as follows.
\[
F=
\begin{tikzpicture}[baseline=-0.5ex,scale=0.05]
\node at (-10,10) {$P_v$};
\draw (-22,5) rectangle (2,15);
\node at (-10,-10) {$P_v$};
\draw (-22,-15) rectangle (2,-5);
\draw (-20,-5) -- (-20,5);
\node at (-12,0) {$\dots$};
\draw (-4,-5) -- (-4,5);
\draw (0,-5) .. controls (0,0) and (10,0) .. (10,-20);
\draw (0,5) .. controls (0,0) and (10,0) .. (10,20);
\draw (-20,-20) -- (-20,-15);
\node at (-10,-18) {$\dots$};
\draw (0,-20) -- (0,-15);
\draw (-20,15) -- (-20,20);
\node at (-10,18) {$\dots$};
\draw (0,15) -- (0,20);
\end{tikzpicture}
\simeq
\begin{tikzpicture}[baseline=-0.5ex,scale=0.05]
\node at (-10,10) {$P_v$};
\draw (-22,5) rectangle (2,15);
\node at (-10,-10) {$P_v$};
\draw (-22,-15) rectangle (2,-5);
\draw (-20,-5) -- (-20,5);
\node at (-10,0) {$\dots$};
\draw (0,-5) -- (0,5);
\draw (0,-15) .. controls (0,-20) and (10,-20) .. (10,0)
              .. controls (10,20) and (0,20) .. (0,15);
\draw (-20,-20) -- (-20,-15);
\node at (-12,-18) {$\dots$};
\draw (-4,-20) -- (-4,-15);
\draw (-20,15) -- (-20,20);
\node at (-12,18) {$\dots$};
\draw (-4,15) -- (-4,20);
\end{tikzpicture}
= P_{v-1}.
\]
Here, $P_v\otimes\id_1-F=P_{v+1}$ holds up to a scalar by
Lemma~\ref{lem:uniqueuptoscalar}.  An easy way to check that the scalar
is $1$ is to calculate traces, which are $2$ for $P_{v+1}$ and $P_v$,
and $4$ for $P_v\otimes\id_1$.
\end{proof}

\section{A converse}

Suppose $\mathcal{C}$ is an $R$-linear spherical category whose objects
are the natural numbers.  Suppose the principal graph of $\mathcal{C}$
is $D_\infty$, in the sense of Definition~\ref{def:principal}.

Further suppose $\mathcal{C}$ has bubble constant $2$, meaning
$$\tr(\id_1)=2\;\id_0.$$
This might follow automatically from the other assumptions.  Indeed,
there is a unique hyperfinite subfactor with principal graph $D_\infty$,
by Popa's result that amenable subfactors are classified by their standard
invariant \cite{MR1278111}.  However nobody has looked into this question
in our context where the planar algebras have no shading and fewer axioms,
and our ring of scalars need not be $\mathbb{C}$.

\begin{theorem}
If $\mathcal{C}$ satisfies the above assumptions then it is isomorphic
to $\overline{\mathcal{B}}$.
\end{theorem}

\begin{proof}
We have an idempotent $P_v$ corresponding to each vertex $v$ in
$D_\infty$.  These are really isomorphism classes, and we can choose a
representative $P_v$ that is an $n$-box, where $n$ is the distance from
$v$ to the base vertex in $\Gamma$.  We are abusing notation, since $P_v$
were idempotents in $\overline{\mathcal{B}}$ in the previous section.

Let our functor $\phi$ take the object $n$ in $\overline{\mathcal{B}}$
to the object $n$ in $\mathcal{C}$, and a crossing in
$\overline{\mathcal{B}}$ to $\id_2-2P_{2'}$ in $\mathcal{C}$.
$$
\phi \colon
\begin{tikzpicture}[baseline=-0.5ex,scale=0.05]
\draw (-5,-10) .. controls (-5,0) and (5,0) .. (5,10);
\draw (5,-10) .. controls (5,0) and (-5,0) .. (-5,10);
\end{tikzpicture}
\;\mapsto\;
\id_2-2P_{2'}.
$$
This rearranges the definition from the previous section of $P_{2'}$
in terms of a crossing to instead define a crossing in terms of $P_{2'}$.

We also insist that $\phi$ is a strong monoidal functor, and preserves
the pivotal structure.  Concretely, this means $\phi$ takes a Brauer
diagram with $k$ crossings to a linear combination of $2^k$ diagrams in
$\mathcal{C}$ that are made of strands and copies of $P_{2'}$.  Here,
the axioms of a spherical category guarantee that we can represent
morphisms in $\mathcal{C}$ by diagrams in a rectangle.

We must show that $\phi$ is well defined.  Obviously it repects the
``bubble equals $2$'' relation.  It remains to show that if we define
a crossing in $\mathcal{C}$ as
$$T = \id_2 - 2 P_{2'},$$
then it is invariant under $90$ degree rotation, satisfies Reidemeister
moves, and gives $A_3=0$.

We have $P_0=\id_0$, which is represented by the empty diagram, and $P_1$
must be $\id_1$, which is represented by a single vertical strand. Next,
$\id_2$ is a direct sum
$$\id_2 = P_1\otimes\id_1 = E + P_2 + P_{2'},$$
where $E\simeq P_0$.  Since $E$ is the unique $2$-box that is isomorphic
to $\id_0$, it must be the Jones idempotent
$$E=
\frac12\;
\begin{tikzpicture}[baseline=-0.5ex,scale=0.05]
\draw (-5,-5) .. controls (-5,0) and (5,0) .. (5,-5);
\draw (-5,5) .. controls (-5,0) and (5,0) .. (5,5);
\end{tikzpicture}
$$
The $2$-box space is freely generated by $E$, $P_2$, and $P_{2'}$.

We list some relations that $P_{2'}$ must satisfy.
\[
EP_{2'}=0
,\qquad
\tr(P_{2'})=1
,\qquad
P_{2'}^\dagger=P_{2'}.
\]
The first came from the definition of direct sum, and implies
that $P_{2'}$ is uncappable.  The trace of $P_{2'}$ is $1$ because
$P_{2'}\otimes\id_1$ is isomorphic to $P_1$, which has trace $2$.
The $180$ degree rotation of $P_{2'}$ is also an uncappable idempotent
$2$-box with trace $1$, so it must equal $P_{2'}$, since no other linear
combination of $E$, $P_2$, and $P_{2'}$ has these properties.

Importantly, $P_{2'}$ also satisfies the leaf relation from
$\overline{\mathcal{B}}$, as shown in Lemma~\ref{lem:leaf}.  To see
this, note that $P_{2'}\otimes\id_1$ is isomorphic to $P_1$, so its
space of endomorphisms is free of rank one.  Thus the two sides of the
leaf relation are equal up to scalars.  To see that the scalar $2$ in
the relation is correct, we could calculate the trace of both sides,
which is $2$.  A similar relation holds any time we have an idempotent
corresponding to a vertex of valence one, or ``leaf'' of a principal
graph, hence the name ``leaf relation''.

The \emph{right partial trace} of an $n$-box is obtained by attaching a
strand that connects the top right and bottom right endpoints.  The right
partial trace of $P_{2'}$ must be $\frac12\id_1$, since that is the
unique $1$-box with trace $1$.  Similarly for the left partial trace.
The left partial trace of the leaf relation implies
$$\rho(P_{2'})^2=\frac14\id_2.$$

We now have enough relations in $P_{2'}$ to prove the necessary relation
in $T$.  We have the following versions of Reidemeister one and two.
$$ET=T,\qquad T^2=\id_2.$$
Let $\rho$ denote the $90$ degree rotation operator on $2$-boxes,
so $\rho(T)$ is as follows.
\[
\rho(T) = 
\begin{tikzpicture}[baseline=-0.5ex,scale=0.05]
\node at (0,0) {$T$};
\draw (-7,-5) rectangle (7,5);
\draw (-5,-10) -- (-5,-5);
\draw (5,5) -- (5,10);
\draw (5,-5) .. controls (5,-10) and (15,-10) .. (15,10);
\draw (-5,5) .. controls (-5,10) and (-15,10) .. (-15,-10);
\end{tikzpicture}.
\]
This satisfies the same versions of Reidemeister one and two
$$E\rho(T)=E,\qquad \rho(T)^2=\id_2.$$

Let $\alpha, \beta, \gamma \in R$  be the coefficients of $\rho(T)$
in terms of our generating idempotents
$$\rho(T) = \alpha E + \beta P_2 + \gamma P_{2'}.$$
Plug this into the three relations
$$E\rho(T)=E
,\qquad
\tr(\rho(T))=2
,\qquad
(\rho(T))^2=\id_2,
$$
to get the equations
$$
\alpha=1
,\qquad
\alpha+2\beta+\gamma=2
,\qquad
\alpha^2=\beta^2=\gamma^2=1.
$$
This system of equations is easy to solve, even if $R$ can have zero
divisors.  We get the unique solution
$$
\rho(T) = E + P_2 - P_{2'} = \id_2 - 2P_{2'} = T.
$$
It is now reasonably safe to represent $T$ diagrammatically as a virtual
crossing, although we cannot yet apply Reidemeister three.

Next we prove $A_3=0$.  The idea is to follow the proof of
Lemma~\ref{lem:leaf} in reverse, this time showing that the leaf
relation implies $A_3=0$.  Take the definiton of $A_3$ exactly as drawn
in Section~\ref{sec:brauer}.  Expand every crossing between the left
two strands to $\id_2-2P_{2'}$, leaving crossings between the right
two strands untouched.  Everything cancels except the expression for
$A_3$ as shown in the proof of Lemma~\ref{lem:leaf}.  This in turn is
equivalent to the leaf relation by adding a crossing between the top
right and bottom left endpoints, and using only Reidemeister one and two.

Now $A_3$ and its dual are both zero, and therefore so is their difference
\[
A_3 - A_3^\dagger =
\begin{tikzpicture}[baseline=-0.5ex,scale=0.15]
\draw (2,-2) .. controls (2,-1) and (-2,-1) .. (-2,2);
\draw (0,-2) .. controls (0,-1) and (1,-1) ..
      (1,0) .. controls (1,1) and (0,1) ..  (0,2);
\draw (-2,-2) .. controls (-2,1) and (2,1) .. (2,2);
\end{tikzpicture}
\;-\;
\begin{tikzpicture}[baseline=-0.5ex,scale=0.15]
\draw (-2,-2) .. controls (-2,-1) and (2,-1) .. (2,2);
\draw (0,-2) .. controls (0,-1) and (-1,-1) ..
      (-1,0) .. controls (-1,1) and (0,1) ..  (0,2);
\draw (2,-2) .. controls (2,1) and (-2,1) .. (-2,2);
\end{tikzpicture}
.
\]
Thus Reidemeister three holds for $T$ in $\mathcal{C}$.

We conclude that $\phi$ is a well-defined functor from
$\overline{\mathcal{B}}$ to $\mathcal{C}$.  It is the obvious bijection on
objects, and the obvious isomorphism on the $0$-box spaces.  Recall that
$\overline{\mathcal{B}}$ and $\mathcal{C}$ both have principal graph
$D_\infty$, and the traces of the idempotents $P_v$ are all units of $R$.
The fact that $\phi$ is an isomorphism now follows from general algebraic
nonsense.

I claim that $\overline{\mathcal{B}}$ has no non-zero negligible
morphisms.  By induction, every $\id_n$ is isomorphic to a direct sum of
idempotents $P_v$.  Any morphism breaks down as a matrix of morphisms
between idempotents that are isomorphic to $P_v$.  A non-zero morphism
must have a non-zero matrix entry.  Thus we can compose by an inclusion
and a projection to get a non-zero isomorphism between idempotents that
are isomorphic to some $P_v$.  Such an isomorphism has non-zero trace.
It follows that $\phi$ is faithful.

Now consider a morphism in $\mathcal{C}$.
It also breaks down as a matrix of morphisms
between idempotents that are isomorphic to bricks $P_v$.
The relevant isomorphisms are all in the image of $\phi$.
The automorphisms of $P_v$ are all in the image of $\phi$.
It follows that $\phi$ is full.
\end{proof}

Note that most of the above proof only used facts about the $2$-box
space, together with the leaf relation in the $3$-box space.  Thus it
shows that if $\mathcal{C}$ has principal graph $D^{(1)}_n$ for $n\ge5$
then it contains a copy of $\overline{\mathcal{B}}$.  This also holds
for $D^{(1)}_4$ when working over $\mathbb{C}$, but I have not looked
into this over $R$.

\section{Embeddings into other planar algebras}

We finish by noting some embeddings of embeddings of
$\overline{\mathcal{B}}$ into other planar algebras.  We will work
over $\mathbb{C}$.  This might not be strictly necessary, but it saves
some work.  In particular, our embeddings automatically map negligible
morphisms to negligible morphisms, since morphisms are negligible if
and only if they are null vectors for the usual sesquilinear form.

\subsection{Affine $A_\infty$}

Define a planar algebra whose diagrams are Temperley-Lieb diagrams
except that a strand can have one or more tags.  A tag is drawn as a
short edge attached to one side of a strand.  Two adjacent tags on the
same side of a strand can cancel, a bubble with no tags evaluates to $2$,
and a bubble with one tag evaluates to $0$.  The final relation says we
can switch a tag to the opposite side of its strand, and multiply the
diagram by $\omega$, where $\omega = \pm 1$.  As usual, we quotient by
negligible morphisms.

The two values of $\omega$ give us the two planar algebras with principal
graph $\tilde{A}_\infty$.  Molander recently gave an equivalent description
of these planar algebras \cite{molander2024skeintheoryaffinesubfactor}.

The fact that the $D_\infty$ subfactor planar algebra lies inside the
$\tilde{A}_\infty$ subfactor planar algebra is well known to experts, at
least in the subfactor context.  A proof appears in Molander's PhD thesis.
An explicit embedding is as follows.
$$
\begin{tikzpicture}[baseline=-0.5ex,scale=0.05]
\draw (-5,-10) .. controls (-5,0) and (5,0) .. (5,10);
\draw (5,-10) .. controls (5,0) and (-5,0) .. (-5,10);
\end{tikzpicture}
\;\mapsto\;
\begin{tikzpicture}[baseline=-0.5ex,scale=0.05]
\draw (-5,-10) .. controls (-5,0) and (5,0) .. (5,-10);
\draw (-5,10) .. controls (-5,0) and (5,0) .. (5,10);
\end{tikzpicture}
\;-\;
\begin{tikzpicture}[baseline=-0.5ex,scale=0.05]
\draw (-5,-10) -- (-5,10);
\draw (5,-10) -- (5,10);
\draw (-5,0) -- (-3,0);
\draw (5,0) -- (3,0);
\end{tikzpicture}
$$

\subsection{Bubbles and confetti}

Define the \emph{bubbles and confetti} planar algebra, whose diagrams
are Temperley-Lieb diagrams except that a strand can end at a univalent
vertex.  A bubble evaluates to two.  A piece of ``confetti'', consisting
of two univalent vertices joined by a strand, evaluates to one.  As usual,
we quotient by negligible morphisms.


This has principal graph consisting of two vertices, an edge between them,
and two loops such that each vertex is adjacent to itself.  The planar
algebra with principal graph $\tilde{A}_\infty$ and $\omega=1$ embeds
into the bubbles and confetti planar algebra as follows.
$$
\begin{tikzpicture}[baseline=-0.5ex,scale=0.05]
\draw (0,-10) -- (0,10);
\draw (0,0) -- (2,0);
\end{tikzpicture}
\;\mapsto\;
\begin{tikzpicture}[baseline=-0.5ex,scale=0.05]
\draw (0,-10) -- (0,10);
\end{tikzpicture}
\;-2\;
\begin{tikzpicture}[baseline=-0.5ex,scale=0.05]
\draw (0,-10) -- (0,-5);
\node at (0,-5) {$\bullet$};
\draw (0,10) -- (0,5);
\node at (0,5) {$\bullet$};
\end{tikzpicture}
$$

\subsection{Principal graph $A_3$}

Consider the Temperley-Lieb algebra with bubble constant $\sqrt{2}$,
modulo negligible morphisms. This has principal graph $A_3$ (not to be
confused with the antisymmetrizer $A_3$ in $\mathcal{B}$).

The following is an embedding from the bubbles and confetti planar
algebra into the $A_3$ planar algebra.
$$
\begin{tikzpicture}[baseline=-0.5ex,scale=0.05]
\draw (0,-5) -- (0,5);
\node at (0,5) {$\bullet$};
\end{tikzpicture}
\mapsto
2^{-\frac{1}{4}}\;
\begin{tikzpicture}[baseline=-0.5ex,scale=0.05]
\draw (-2,-5) .. controls (-2,6) .. (0,6) .. controls (2,6) .. (2,-5);
\end{tikzpicture}
$$
This is a slightly more general notion of ``embedding'' since it doubles
the strands.  To be precise, we could think of it as a functor that acts
as mutiplication by $2$ on objects.

\subsection{Embeddings into its own Cauchy completion}

Suppose $k \ge 1$.  There is a functor from $\overline{\mathcal{B}}$ to
its own Cauchy completion that maps $1$ to the image of $P_k$.  A crossing
the diagram where $k$ parallel strands cross another $k$ parallel strands,
with domain and codomain restricted to $P_k \otimes P_k$.


\begin{thebibliography}{1}

\bibitem{MR1437496}
Dietmar Bisch and Vaughan Jones.
\newblock Algebras associated to intermediate subfactors.
\newblock {\em Invent. Math.}, 128(1):89--157, 1997.

\bibitem{MR1733737}
Dietmar Bisch and Vaughan Jones.
\newblock Singly generated planar algebras of small dimension.
\newblock {\em Duke Math. J.}, 101(1):41--75, 2000.

\bibitem{MR4374438}
V.~F.~R. Jones.
\newblock Planar algebras, {I}.
\newblock {\em New Zealand J. Math.}, 52:1--107, 2021 [2021--2022].

\bibitem{MR3420509}
G.~I. Lehrer and R.~B. Zhang.
\newblock The {B}rauer category and invariant theory.
\newblock {\em J. Eur. Math. Soc. (JEMS)}, 17(9):2311--2351, 2015.

\bibitem{molander2024skeintheoryaffinesubfactor}
Melody Molander.
\newblock Skein theory for affine a subfactor planar algebras, 2024.

\bibitem{MR2559686}
Scott Morrison, Emily Peters, and Noah Snyder.
\newblock Skein theory for the {$D_{2n}$} planar algebras.
\newblock {\em J. Pure Appl. Algebra}, 214(2):117--139, 2010.

\bibitem{MR1278111}
Sorin Popa.
\newblock Classification of amenable subfactors of type {II}.
\newblock {\em Acta Math.}, 172(2):163--255, 1994.

\end{thebibliography}

\end{document}